\documentclass[11pt,twoside, final]{amsart}
 
\copyrightinfo{0}{Iranian Mathematical Society}
\pagespan{1}{\pageref*{LastPage}}
\usepackage{etoolbox,lastpage}
\commby{}
\date{\scriptsize   Received: , Accepted: .}

\usepackage[pagebackref=true,colorlinks,linkcolor=blue,citecolor=blue,final]{hyperref}
\usepackage{amsmath, amsthm, amscd, amsfonts, amssymb, graphicx, color}
\usepackage{txfonts} 
\usepackage{mathrsfs}
\usepackage{graphicx}
\usepackage[colorlinks]{hyperref}
\usepackage[all]{xy}
\newtheorem{theorem}{Theorem}[section]
\newtheorem{lemma}[theorem]{Lemma}
\newtheorem{proposition}[theorem]{Proposition}
\newtheorem{corollary}[theorem]{Corollary}
\theoremstyle{definition}
\newtheorem{definition}[theorem]{Definition}

\theoremstyle{remark}
\newtheorem{remark}[theorem]{Remark}
\numberwithin{equation}{section}
 
\begin{document}
\title[The Cozero part of the pointfree version]{The Cozero part of the pointfree version of $C_c (X)$}
\author[Estaji]{Ali Akbar Estaji$^{1*}$}
\address[Ali Akbar Estaji]{Faculty of Mathematics and Computer Sciences, Hakim Sabzevari University, Sabzevar, Iran.}
\email{aaestaji@hsu.ac.ir}

\author[Taha]{Maryam Taha}
\address[Maryam Taha]{Faculty of Mathematics and Computer Sciences, Hakim Sabzevari University, Sabzevar, Iran.}
\email{tahamaryam872@gmail.com}

\thanks{$^*$Corresponding author}

 \maketitle

\begin{abstract} 
Let $\mathcal C_{c}(L):= \{\alpha\in \mathcal{R}(L) \mid R_{\alpha} \, \text{ is a countable subset of } \,  \mathbb R \}$,  where $R_\alpha:=\{r\in\mathbb R \mid {\mathrm{coz}}(\alpha-r)\neq\top\}$ for every $\alpha\in\mathcal R  (L).$  By using idempotent elements, it is going to  prove that  ${{\mathrm{Coz}}}_c[L]:= \{{\mathrm{coz}}(\alpha)  \mid  \alpha\in\mathcal{C}_c  (L) \}$ is a  $\sigma$-frame for every completely regular frame $L,$ and from this,  we conclude  that it is regular, paracompact, perfectly normal and an Alexandroff algebra frame such that each cover of it is shrinkable. Also, we show that $L$ is a zero-dimensional frame if and only if $ L$ is a $c$-completely regular frame. 
\end{abstract}
\section{Introduction}
Some authors, such as Ball and Hager \cite{Ball1991}, studied
the ring
$\mathcal R(L)$
of real-valued continuous functions on a
frame, namely the pointfree version of the ring
$C(X)$,
 prior to 1996.
Banaschewski \cite{Banaschewski1997-2} did
a systematic study of
the ring of real-valued continuous functions in pointfree topology 
in 1996.
 Also, to read more about 
 $\mathcal R(L)$, you can see \cite{BallWalters2002, Banaschewski(1997)1, Banaschewski(1996),  Stone1982, PicadoPultr}.

The largest subring of $C(X)$,
 whose elements are countable images, is $C_c(X)$,
which leads to research on the pointfree version of its, that is,
 $\mathcal{R}_c(L)$, by Estaji
et al.
\cite{Sarpoushi-1, Sarpoushi-2,  Sarpoushi-thesis, Elyasi}. 
Estaji at al. \cite{estaj2}
defined
$R_\alpha:=\{r\in\mathbb{R}\mid~\mathrm{coz}(\alpha-r)\neq\top\}$
for every
$\alpha\in \mathcal R(L)$.
The frame of open sets of a space $X$ is called $\mathcal{O}(\mathbb{X})$.
 Aliabadi and Mahmodi \cite{Aliabad} studied $\mathrm{pim}(\alpha):=\bigcap\{\omega\in\mathcal{O}(\mathbb{R})\mid~\alpha(\omega)=\top\}$ for every
$\alpha\in \mathcal R(L)$, and proved that $\mathrm{pim}(\alpha)=R_\alpha$.

 By using the definition of
$R_\alpha$,  
authors of \cite{EstajiTaha0} introduced  another  pointfree version of
$C_c(X)$ via the range of  functions  
and denoted it by
$\mathcal C_{c}(L)$,
which is more compatible with the functions  of countable range. 
Also, in  \cite{EstajiTaha1}, it was proved that  every radical ideal in it is an absolutely
convex ideal.
Characterization of the clean elements of
$\mathcal R(L)$ has been done
and 
 proved when
$\mathcal C_c(L)$
is clean (see \cite{EstajiTaha2}).

The cozero map plays a key
role in a way not usually encountered in ordinary topology; in fact, its
main essence is that it provides a pointfree way of dealing with important
point-dependent classical arguments.
The collection of all cozero elements of $L$ is called the \textit{ cozero part} and  is usually denoted by ${\mathrm{Coz}}[L]$.
General properties of cozero elements  and ${\mathrm{Coz}}[L]$ can be found in \cite{Banaschewski2009,Banaschewski(1996)}.
Cozero elements play a particular role in $\mathcal R(L)$, $\mathcal{R}_c(L)$, and $\mathcal C_{c}(L)$.
The cozero part of a frame $L$ is seen to be a  sublattice of $L$, and also, ${\mathrm{Coz}}[L]$ is a $\sigma$-frame.

The main purpose of this article is to show that
${{\mathrm{Coz}}}_c[L]$ is a $\sigma$-frame for every completely regular frame $L$ and that
$L$ is a zero-dimensional frame if and only if $ L$
is a $c$-completely regular frame.

\section{Preliminaries}
A complete lattice $L$ 
is said to be a \textit{frame} if 
$a\wedge\,\bigvee B=\bigvee_{b\in B}(a\wedge b),$
for any 
$a\in L$ and $B\subseteq L$. 
We denote the \textit{top} element and the \textit{bottom} element of a frame 
$L$ by $\top$ and 
$\bot$, respectively. For every element
$a$
of a frame
$L$, set
 $\uparrow  \hspace{-1mm}a := \{ b \in L \mid a \leq b \}$ and $ a \to b :=\{ x \in L \mid a \wedge x \leq b\}$ for every $a, b\in L.$
The \textit{pseudocomplement} of $a \in L$ is the element $a^* := a \to \bot =\bigvee\{x \in L \mid x\wedge a = \bot\}.$
If $a \vee a^* =\top$, then $a$ is said to be \textit{complemented} and the complement of $a$ denoted by $a'$.

 
Recall  from \cite{Banaschewski1997-2}  that
 the \textit{frame of  reals}
 is the frame
$\mathcal{L}(\mathbb{R})$
generated by
all ordered pairs
$(p,q)$
with $p, q \in \mathbb{Q}$,
subject to the following relations:
\begin{enumerate} 
\item[{\rm (R1)}]  
$(p,q)\wedge (r,s)=(p\vee r,q\wedge s)$, 
\item[{\rm (R2)}]  
$(p,q)\vee (r,s)=(p,s)$, whenever $p\leq r<q\leq s,$ 
\item[{\rm (R3)}]  
$(p,q)=\bigvee \{ (r,s) \mid  p<r<s<q\},$ 
\item[{\rm (R4)}] 
 $\top=\bigvee \{ (p,q) \mid  p,q\in{\mathbb Q}\}$.
\end{enumerate}

A \textit{ continuous real function}\index{continuous real function on a frame} on a frame is a frame homomorphism $\mathcal{L}(\mathbb{R}) \rightarrow L$.
The set of all continuous real functions on a frame $L$ 
 is denoted by $\mathcal{R} (L)$.  
We recall from \cite{Banaschewski1997-2}  that
the set $\mathcal{R}(L)$ with the operator $\diamond\in \{+,\cdot, \wedge, \vee\}$ defined by
\[ \begin{aligned}
(\alpha \diamond \beta)(p,q)=\bigvee\big\{\alpha(r,s)\wedge \beta(u,v)\mid   <r,s>\diamond <u,v>\subseteq <p,q>\big\}  
\,  \text{(}\forall\alpha, \beta \in \mathcal{R}(L)  \text{)}
\end{aligned} \]
is an $f$-ring. Here, $<p,q>:=\{r\in \mathbb{R}\mid  p<r<q\}$ for every $p,q\in \mathbb{Q}$, and 
\[ \begin{aligned}
<r,s>\diamond <u,v>=\big\{y\diamond z\mid  y\in <r,s>, \, z\in <u,v>\big\}.
\end{aligned} \]  

To facilitate the work, for every $p,q\in \mathbb{Q},$ we put 
 \begin{center}
$( p , -):= \bigvee \limits_{\substack{  r,s>p,\\ r,s\in\mathbb{Q}}} (r,s) 
$
and
$ ( - , q):=\bigvee \limits_{\substack{  r,s<q,\\ r,s\in\mathbb{Q}}} (r,s)
.$
 \end{center}  
For  $\alpha \in \mathcal{R} (L),$
$ \alpha( - , 0) \vee\alpha (0 ,-)$
is called \textit{ cozero element} of $L$
and denoted by ${\mathrm{coz}}(\alpha).$ 
In {\em\cite{{Banaschewski1997-2},{PicadoPultr}}},
a frame is completely regular if and only if it is generated by its
cozero elements, that is, for every $x\in L $,  
there exists $ \{\alpha_i\}_{i\in I}\subseteq   \mathcal{R} (L)$ such that 
$$x=\bigvee_{{\mathrm{coz}}(\alpha_i)\leq x} {\mathrm{coz}}(\alpha_i)\,.$$

Apart from frames, we shall also consider  $\sigma$-frames, defined as below.
A \linebreak \textit{ $\sigma$-frame}  is a  distributive lattice $L$ in which any countable subset has a join such that
\[
 a\wedge \bigvee S=\bigvee_{s\in S}(a\wedge s),
 \]
for all $a \in L$ and any countable subset  $S\subseteq L$. 
For example, 
$$
{\mathrm{Coz}}[L]:= \{{\mathrm{coz}}(\alpha)  \mid \alpha\in \mathcal{R}(L) \}
$$
 is a sub-$\sigma$-frame of $L$ (see 
 \cite[Corollary 2]{Banaschewski(1996)}). 

We recall from \cite{estaj2} that for each $\alpha\in\mathcal{R}(L)$, 
 $R_\alpha:=\{r\in\mathbb{R}\mid~\mathrm{coz}(\alpha-r)\neq\top\}$ 
and also, 
$$
\mathcal C_{c}(L):= \{\alpha\in \mathcal{R}(L) \mid
R_{\alpha} \, \text{ is a countable subset of } \,  \mathbb R  \}
$$
  as the pointfree topology version of the ring
$C_c(X)$ (for more details see \cite{EstajiTaha0, EstajiTaha1, EstajiTaha2})
and we set 
$$
\mathrm{Coz}_c[L] := \{{\mathrm{coz}}(\alpha)  \mid  \alpha\in\mathcal{C}_c(L) \}.$$ 
 
\section{Algebraic operations of  the idempotent elements of  $ \mathcal{R}(L)$} 
It is written 
$B(A)$
for the set of all the idempotent elements of a given ring  $A$. If $e\in B(\mathcal{R}(L))$, then ${\mathrm{coz}}(e)$ is complemented.
For  $e_1, e_2\in B(\mathcal{R}(L))$ in this section,
 we check  
$(e_1\diamond e_2)(-,x)$
and
$(e_1\diamond e_2)(x,-)$
for every $x\in \mathbb{Q}$,
where
$\diamond\in\{+, -, \cdot, \vee, \wedge\}.$
\begin{remark} \label{remark3}
Let $e $ be an  idempotent element  of  $ \mathcal{R}(L)$. 
Then the following statements are  true: 
\begin{enumerate} 
\item[{\rm (1)}]  
  For every $x\in \mathbb{Q}$, 
\[ \begin{aligned}
  e(x,-)=  
  \begin{cases}
\top, & x<0, \\
 {{\mathrm{coz}}}(e),& 0\leq x<1,\\
\bot,   & x \geq 1,
\end{cases} \,\,\, \text{and} \,\,\, \,\,\,
  e(-,x)=  
  \begin{cases}
\bot ,& x\leq 0 ,\\
 \left({{\mathrm{coz}}}(e)\right)' ,& 0< x\leq1,\\
\top,   & x > 1.
\end{cases}
 \end{aligned} \]
 \item[{\rm (2)}]  
If $x,k\in \mathbb{Q}$ with $k>0$, then   

\[ \begin{aligned}
  ke(x,-)=  
  \begin{cases}
\top, & x<0, \\
 {{\mathrm{coz}}}(e), & 0\leq x<k,\\
\bot ,  & x \geq k,
\end{cases} \,\,\, \text{and} \,\,\, \,\,\,
  ke(-,x)=  
  \begin{cases}
\bot, & x\leq 0 ,\\
 \left({{\mathrm{coz}}}(e)\right)', & 0< x\leq k,\\
\top ,  & x > k.
\end{cases}
 \end{aligned} \]
 \item[{\rm (3)}]  
 If $x,k\in \mathbb{Q}$ with $k<0$, then   

\[ \begin{aligned}
  ke(x,-)=  
  \begin{cases}
\bot ,& x\geq 0, \\
 \left({{\mathrm{coz}}}(e)\right)' ,& k\leq x<0,\\
\top,  & x< k,
\end{cases} \,\,\, \text{and} \,\,\, \,\,\,
  ke(-,x)=  
  \begin{cases}
\top, & x> 0 ,\\
 {{\mathrm{coz}}}(e) , & k< x\leq 0,\\
\bot ,  & x \leq k.
\end{cases}
 \end{aligned} \]
\end{enumerate} 
\end{remark}
\begin{lemma}    \label{prime5}
If $e_1, e_2\in B(\mathcal{R}(L))$, then  for every $x\in \mathbb{Q}$, we have
\[ \begin{aligned} 
(e_1+e_2)(x,-)=   
\begin{cases}
\top & x<0, \\
 {{\mathrm{coz}}}(e_1+e_2), & 0\leq x<1,\\
  {{\mathrm{coz}}}(e_1) \wedge {{\mathrm{coz}}}(e_2), & 1\leq x < 2,\\
\bot,   & x \geq 2,
\end{cases}
\end{aligned} \]
and 
\[ \begin{aligned} 
(e_1+e_2)(-,x)=   
\begin{cases}
\bot, & x\leq 0, \\
 \big({{\mathrm{coz}}}(e_1)\big)' \wedge \big({{\mathrm{coz}}}(e_2) \big)',
 & 0< x\leq1,\\
  \big({{\mathrm{coz}}}(e_1)\big)' \vee \big({{\mathrm{coz}}}(e_2) \big)' ,& 1< x \leq 2,\\
\bot ,  & x > 2.
\end{cases}
\end{aligned} \]
\end{lemma}
\begin{proof} 
For every $x\in \mathbb{Q}$, if 
$x<0$, then $(e_1+e_2)(x,-)= \top$, 
since $ e_1+e_2\geq \mathbf{0}$. 
It is evident that $ {{\mathrm{coz}}}(e_1+e_2)=(e_1+e_2)(0,-)$. 
For every $x,r\in \mathbb{Q}$, we define 
 $H_{(x,r)}=e_1(r,-)\wedge e_2(x-r,-)$. 
  Let $x\in \mathbb{Q}$ with $0<x<1$ be given. 
  Then there exist $r, s\in \mathbb{Q}$ such that 
  $0<r<x<s<1$, which implies that 
   $H_{(x,x-s)}= {{\mathrm{coz}}}(e_2)$,  
  $H_{(x,s)}= {{\mathrm{coz}}}(e_1)$, and 
  $H_{(x,r)}={{\mathrm{coz}}}(e_1)\wedge  {{\mathrm{coz}}}(e_2)$.
  Therefore, 
\[ \begin{aligned}
  (e_1+e_2)(x,-)&=   {{\mathrm{coz}}}(e_2)\vee  {{\mathrm{coz}}}(e_1)\vee \big( {{\mathrm{coz}}}(e_1)\wedge  {{\mathrm{coz}}}(e_2)\big)\\
  &=  {{\mathrm{coz}}}(e_1)\vee  {{\mathrm{coz}}}(e_2)\\
  &= {{\mathrm{coz}}}(e_1+e_2) 
 \end{aligned} \]
 for every $x\in \mathbb{Q}$ with $0\leq x<1$. 
 Let $x\in \mathbb{Q}$ with $1\leq x<2$ be given.   
 Then for every $r\in \mathbb{Q}$, 
 \[ \begin{aligned}    
 & x-r<0 \Rightarrow 1\leq x<r
  \Rightarrow H_{(x,r)}=\bot, \\
 & 0\leq x-r <1 
   \Rightarrow \left(0\leq r <1 \text{ and } 0\leq x-r <1 \right)
    \Rightarrow  H_{(x,r)}= {{\mathrm{coz}}}(e_1)\wedge  {{\mathrm{coz}}}(e_2),  \text{ and } \\
&   x-r\geq 1  \Rightarrow  H_{(x,r)}=\bot.
 \end{aligned} \]
 Hence, 
 \[ \begin{aligned}
  (e_1+e_2)(x,-)=   {{\mathrm{coz}}}(e_1)\wedge  {{\mathrm{coz}}}(e_2) 
 \end{aligned} \]
  for every $x\in \mathbb{Q}$ with $1\leq x<2$. Let $x\in \mathbb{Q}$ with $ x\geq 2$ be given.   
 Then for every $r\in \mathbb{Q}$,  if  
$    0\leq x-r <1$ or $x-r<0$, then $r>1$ or $2\leq x<r$, which implies that 
$ H_{(x,r)}=\bot$. 
Therefore,  for every $x\in \mathbb{Q}$ with $2\leq x $, we have  
$
  (e_1+e_2)(x,-)=    \bot. 
$ 
\end{proof}
\begin{lemma}    \label{prime80}
If $e_1, e_2\in B(\mathcal{R}(L))$, then  for every $x\in \mathbb{Q}$, it holds that
\[ \begin{aligned} 
(e_1e_2)(x,-)=   
\begin{cases}
\top, & x<0, \\
{{\mathrm{coz}}}(e_1) \wedge {{\mathrm{coz}}}(e_2), &  0\leq x<1,\\
\bot, & x\geq 1,
\end{cases}
\end{aligned} \]
and 
\[ \begin{aligned} 
(e_1e_2)(-,x)=   
\begin{cases}
\bot, & x\leq 0, \\
\big({{\mathrm{coz}}}(e_1) \big)'\vee \big({{\mathrm{coz}}}(e_2) \big)', & 0<x\leq 1,\\
\top, &  x>1.
\end{cases}
\end{aligned} \]
\end{lemma}
\begin{proof} 
For every $x\in \mathbb{Q}$, if 
$x<0$, then $(e_1e_2)(x,-)= \top$, 
since $ e_1e_2\geq \mathbf{0}$. 
For every 
$x,s\in \mathbb{Q}$, 
we define 
 $H_{(x,s)}=e_1(s,-)\wedge e_2(\frac{x}{s},-)$.
Let
$0<s\in \mathbb{Q}$.
Then for
$x=0$,
we have
 \[ \begin{aligned}    
 & 0<s<1  \Rightarrow  H_{(x,s)}={{\mathrm{coz}}}(e_1)\wedge  {{\mathrm{coz}}}(e_2), \\
 & s\geq 1    \Rightarrow 
    H_{(x,s)}= \bot.
 \end{aligned} \] 
Therefore, 
$(e_1e_2)(0,-)=\big({{\mathrm{coz}}}(e_1)\wedge  {{\mathrm{coz}}}(e_2)\big)\vee \bot= {{\mathrm{coz}}}(e_1)\wedge  {{\mathrm{coz}}}(e_2)$. 
 
Let 
$0<x<1$.
Then 
 \[ \begin{aligned}    
 &x<s<1  \Rightarrow  H_{(x,s)}={{\mathrm{coz}}}(e_1)\wedge  {{\mathrm{coz}}}(e_2), \\
 & s<x  \text{ and } s>1 \Rightarrow 
    H_{(x,s)}= \bot.
 \end{aligned} \] 
Hence, 
$(e_1e_2)(x,-)=\big({{\mathrm{coz}}}(e_1)\wedge  {{\mathrm{coz}}}(e_2)\big)\vee \bot= {{\mathrm{coz}}}(e_1)\wedge  {{\mathrm{coz}}}(e_2)$. 

Let 
$x\geq1$. Then
\[ \begin{aligned}    
 &x\geq s  \Rightarrow  H_{(x,s)}=\bot, \\
 & s> x  \Rightarrow 
    H_{(x,s)}= \bot.
 \end{aligned} \] 
 Thus,
$(e_1e_2)(x,-)= \bot.$

 For every $x\in \mathbb{Q}$, if 
$x<0$, then $(e_1e_2)(x,-)= \bot$.
For every 
$x,s\in \mathbb{Q}$, 
we define 
 $H_{(s,x)}=e_1(-,s)\wedge e_2(-,\frac{x}{s})$. 
Let
$0<s\in \mathbb{Q}$.
For
$x=0$,
we have
 \[ \begin{aligned}    
 & 0<s\leq1 \Rightarrow  H_{(x,s)}=\big({{\mathrm{coz}}}(e_1) \big)'\wedge  \bot=\bot, \\
 & s> 1    \Rightarrow 
    H_{(x,s)}= \top\wedge\bot=\bot.
 \end{aligned} \] 
Therefore, 
$(e_1e_2)(-,x)= \bot$.

 Let 
$0<x\leq 1$.
Then 
 \[ \begin{aligned}    
 & 0<s\leq x  \Rightarrow  H_{(s,x)}=\big({{\mathrm{coz}}}(e_1) \big)'\wedge\top=\big({{\mathrm{coz}}}(e_1) \big)' , \\
  & x<s\leq1  \Rightarrow  H_{(s,x)}=\big({{\mathrm{coz}}}(e_1) \big)'\wedge   \big({{\mathrm{coz}}}(e_2) \big)', \\
 & s> 1    \Rightarrow 
    H_{(s,x)}= \top\wedge\big({{\mathrm{coz}}}(e_2) \big)'=\big({{\mathrm{coz}}}(e_2) \big)'.
  \end{aligned} \] 

 Thus,
 $$
\begin{array}{lll}
(e_1e_2)(-,x)  &= \big({{\mathrm{coz}}}(e_1) \big)'\vee\big(\big({{\mathrm{coz}}}(e_1) \big)'\wedge\big({{\mathrm{coz}}}(e_2) \big)'\big)\vee \big({{\mathrm{coz}}}(e_2) \big)' \\[2mm]
  &= \big({{\mathrm{coz}}}(e_1) \big)'\vee \big({{\mathrm{coz}}}(e_2) \big)'.\\
\end{array}
$$
 Let 
$x> 1$. 
Then
\[ \begin{aligned}    
 & s\leq 1  \Rightarrow  H_{(s,x)}=\big({{\mathrm{coz}}}(e_1) \big)'\wedge\top=\big({{\mathrm{coz}}}(e_1) \big)' , \\
  &x<s  \Rightarrow  H_{(s,x)}=\top\wedge   \big({{\mathrm{coz}}}(e_2) \big)'=\big({{\mathrm{coz}}}(e_2) \big)', \\
 &1<s<x    \Rightarrow 
    H_{(s,x)}= \top.
  \end{aligned} \] 
  Therefore,
$(e_1e_2)(-,x)=\big({{\mathrm{coz}}}(e_1) \big)'\vee\big({{\mathrm{coz}}}(e_2) \big)'\vee\top=\top.$
\end{proof}
Similar to the process of proving the above lemma, the next lemma was proved.
\begin{lemma}    \label{prime60}
If 
$e_1, e_2\in B(\mathcal{R}(L))$, 
then  for every $x\in \mathbb{Q}$, we have
\[ \begin{aligned} 
(e_1\vee e_2)(x,-)=   
\begin{cases}
\top, & x<0, \\
{{\mathrm{coz}}}(e_1+e_2), & 0\leq x<1,\\
\bot,   & x \geq 1,
\end{cases}
\end{aligned} \]

\[ \begin{aligned} 
(e_1\vee e_2)(-,x)=   
\begin{cases}
\bot ,& x\leq 0, \\
 \big({{\mathrm{coz}}}(e_1)\big)' \wedge \big({{\mathrm{coz}}}(e_2) \big)',
 & 0< x\leq1,\\
\top,   & x > 1,
\end{cases}
\end{aligned} \]
\[ \begin{aligned} 
(e_1\wedge e_2)(x,-)=   
\begin{cases}
\top, & x<0, \\
 {{\mathrm{coz}}}(e_1\wedge e_2), & 0\leq x<1,\\
\bot,   & x \geq 1,
\end{cases}
\end{aligned} \]

and 

\[ \begin{aligned} 
(e_1\wedge e_2)(-,x)=   
\begin{cases}
\bot, & x\leq 0, \\
   \big({{\mathrm{coz}}}(e_1)\big)' \vee \big({{\mathrm{coz}}}(e_2) \big)',  & 0< x \leq 1,\\
\top,   & x >1.
\end{cases}
\end{aligned} \]
\end{lemma}
\section{Dedekind cuts for $\mathbb{R}$ by  idempotent  elements and prime ideals} 
In article \cite{Aliabad}, 
$ A_{p}(\alpha)$ and $B_{p}(\alpha)$
were introduced. 
We found that it is better to use $P_{u}(\alpha)$ and $P_{l}(\alpha)$ symbols instead of them, which would  be easier for the reader to understand the concepts.
They are given below.
For every 
$\alpha\in \mathcal{R}(L)$ and every prime ideal 
$P $ in $L$, let 
\begin{center}  
  $ P_{u}(\alpha):=\{ x\in \mathbb{Q}\mid \alpha(x,-) \in P\}$ \,\,\,\,and\,\,\,\, $ P_{l} (\alpha):=\{x\in \mathbb{Q}\mid  \alpha(-,x)\in P\}$.
 \end{center} 
To express our arguments, we need the following lemma and result from  \cite{Aliabad}.
\begin{lemma}  \label{prime15} {\rm \cite{Aliabad} }
For every 
$\alpha\in \mathcal{R}(L)$ and every prime ideal 
$P $ in $L$, the following statements are  true: 
\begin{enumerate} 
\item[{\rm (1)}]   $ P_{u}(\alpha)\cup   P_{l}(\alpha)=\mathbb{Q}$.
\item[{\rm (2)}]  Any element of $ P_{u}(\alpha)$ is an upper bound of $P_{l}(\alpha)$ and any element of $P_{l}(\alpha)$ is a lower bound of $ P_{u}(\alpha)$.
\item[{\rm (3)}]   ${\uparrow} P_{u}(\alpha)= P_{u}(\alpha)$ 
and  ${\downarrow} P_{l}(\alpha)= P_{l}(\alpha)$. 
\end{enumerate}
\end{lemma}
\begin{corollary}   \label{prime16} {\rm \cite{Aliabad} }
Let $P $ be a prime ideal 
 of $L$ and let $\alpha\in \mathcal{R}(L)$ be given. Then the following statements are equivalent:
\begin{enumerate} 
\item[{\rm (1)}]  $ \inf P_{u}(\alpha)\in \mathbb{R}$.
\item[{\rm (2)}]   $P_{l}(\alpha)\neq\emptyset\neq P_{u}(\alpha)$.
\item[{\rm (3)}]  $ \sup P_{l}(\alpha)\in \mathbb{R}$.
\item[{\rm (4)}]  $ \inf P_{u}(\alpha)=\sup P_{l}(\alpha)\in \mathbb{R}$.     
\end{enumerate}
\end{corollary}
\begin{corollary}   \label{prime17} 
Let $P $ be a prime ideal 
 of $L$ and let  $\alpha\in \mathcal{R}(L)$ be given. 
Then $ \inf P_{u}(\alpha)\in \mathbb{Q}$ if and only if $ \inf P_{u}(\alpha)\in  P_{u}(\alpha)  $ if and only if 
 $ \left|P_{u}(\alpha)\cap   P_{l}(\alpha)\right|=1$.
\end{corollary}
\begin{proof}
By Corollary    \ref{prime16}, it is evident.  
\end{proof}
\begin{lemma}    \label{prime21} 
Let $P $   be a prime ideal 
 of $L$. If $e$   be an idempotent element of  $ \mathcal{R}(L)$, then  the following statements are  true: 
 \begin{enumerate} 
\item[{\rm (1)}]   
 $\inf  P_{u}(e)= \sup P_{l} (e)\in  P_{u}(e)\cap   P_{l}(e) $ and  
\[ \begin{aligned}  
\inf  P_{u}(e)=  
\begin{cases}
0,& {{\mathrm{coz}}}(e)\in P,\\
1,& {{\mathrm{coz}}}(e)\not\in P.
\end{cases}
 \end{aligned} \]
 \item[{\rm (2)}] For every $k\in \mathbb{Q}$, 
  $\inf  P_{u}(ke)=k\inf  P_{u}(e)$.
 \end{enumerate}
\end{lemma}
\begin{proof} (1). 
It is evident that 
 ${{\mathrm{coz}}}(e)\in P$ or $\big({{\mathrm{coz}}}(e) \big)'\in P$. 
 If ${{\mathrm{coz}}}(e)\in P$, then 
\begin{center}  
  $ P_{u}(e)=\{ x\in \mathbb{Q}\mid x\geq 0\}$\,\,\,\,and\,\,\,\, $ P_{l} (e)=\{x\in \mathbb{Q}\mid x\leq 0\}$.
 \end{center} 
 Therefore,  
\[ \begin{aligned}  
\inf  P_{u}(e)= \sup P_{l} (e)=0\in   P_{u}(e)\cap   P_{l}(e).
\end{aligned} \]
If  $\big({{\mathrm{coz}}}(e) \big)'\in P$,  then 
\begin{center}  
  $ P_{u}(e)=\{ x\in \mathbb{Q}\mid x\geq 1\}$ and $ P_{l} (e)=\{x\in \mathbb{Q}\mid x\leq 1\}$.
 \end{center} 
Hence,  
\[ \begin{aligned}  
\inf  P_{u}(e)= \sup P_{l} (e)=1\in   P_{u}(e)\cap   P_{l}(e).
\end{aligned} \]
(2). Let  $k\in \mathbb{Q}$ with $k>0$ be given. Then 
\[ \begin{aligned}    
x\in  P_{u}(ke)
\Leftrightarrow ke(x,-)\in P
\Leftrightarrow  e(\frac{x}{k},-)\in P
\Leftrightarrow x\in k P_{u}(e) ,
\end{aligned} \]
which implies that 
  $\inf  P_{u}(ke)=k\inf  P_{u}(e)$.

Let  $k\in \mathbb{Q}$ with $k<0$ be given. Then 
\[ \begin{aligned}    
x\in  P_{u}(ke)
\Leftrightarrow ke(x,-)\in P
\Leftrightarrow  e(-,\frac{x}{k})\in P
\Leftrightarrow x\in k P_{l}(e) ,
\end{aligned} \]
which implies that 
\[ \begin{aligned} 
  \inf  P_{u}(ke)=
  \inf  k P_{l}(e) =k \sup P_{l}(e)=
  k\inf  P_{u}(e).
 \end{aligned} \] 
\end{proof}
In \cite[Theorem 2.23]{Aliabad}, the prime ideal
$P$ is considered as  countably $\bigvee$-complete, which here, since our discussion is about idempotent elements, we are able to remove this condition and express it as follows.
\begin{proposition}    \label{prime20}
Let $P $   be a prime ideal 
 of $L$.   If $e_1, e_2\in B(\mathcal{R}(L))$, then 
$ 
\inf  P_{u}(e_1\diamond e_2)= \sup P_{l} (e_1\diamond e_2) \in   P_{u}(e_1\diamond e_2)\cap   P_{l}(e_1\diamond e_2) 
$
and
$\inf  P_{u}(e_1\diamond e_2)=\inf  P_{u}(e_1)\diamond\inf  P_{u}(e_2)$
for every $\diamond\in\{ +, \vee, \wedge, \cdot\}.$ 
\end{proposition}
\begin{proof}
We only prove it for $\diamond=+$; the rest will be proved similarly. It is evident that \linebreak $P_{u}(e_1+e_2)\cap (-\infty, 0)=\emptyset$ 
and $P_{l}(e_1+e_2)\cap (0,+\infty)=\emptyset$. 

 Case (1): If   ${{\mathrm{coz}}}(e_1)\in P$ and ${{\mathrm{coz}}}(e_2) \in P$, then 
  ${{\mathrm{coz}}}(e_1+e_2)={{\mathrm{coz}}}(e_1)\vee {{\mathrm{coz}}}(e_2) \in P,$
  which implies that 
  $P_{u}(e_1+e_2)=\{x\in \mathbb{Q}\mid x\geq 0\}$ 
  and 
   $P_{l}(e_1+e_2)=\{x\in \mathbb{Q}\mid x\leq 0\}$.  
  Therefore,    
\[ \begin{aligned}  
\inf  P_{u}(e_1+e_2)= \sup P_{l} (e_1+e_2)=0\in   P_{u}(e_1+e_2)\cap   P_{l}(e_1+e_2)
\end{aligned} \]
and 
$\inf  P_{u}(e_1+e_2)=\inf  P_{u}(e_1)+\inf  P_{u}(e_2).
$

  Case (2): If   ${{\mathrm{coz}}}(e_1)\in P$ and $\big({{\mathrm{coz}}}(e_2) \big)'\in P$, then 
 ${{\mathrm{coz}}}(e_1+e_2) \vee \big({{\mathrm{coz}}}(e_2) \big)'=\top$, 
 which implies that 
 ${{\mathrm{coz}}}(e_1+e_2) \not \in P$. 
 Hence, 
  $P_{u}(e_1+e_2)\cap [0, 1)\cap \mathbb{Q}=\emptyset$. 
  Since \linebreak${{\mathrm{coz}}}(e_1)\wedge{{\mathrm{coz}}}(e_2)\leq  {{\mathrm{coz}}}(e_1)\in P$, we conclude that  
  $[1,2)\cap \mathbb{Q}\subseteq P_{u}(e_1+e_2)$. 
  It is evident that $[2,+\infty)  \subseteq P_{u}(e_1+e_2)$ and implies that  $ P_{u}(e_1+e_2)=[1,+\infty)\cap \mathbb{Q}$. 
  A similar proof shows that  $ P_{l}(e_1+e_2)=(-\infty, 1]\cap \mathbb{Q}$. 
  Therefore,    
\[ \begin{aligned}  
\inf  P_{u}(e_1+e_2)= \sup P_{l} (e_1+e_2)=1\in   P_{u}(e_1+e_2)\cap   P_{l}(e_1+e_2) 
\end{aligned} \]
and 
$\inf  P_{u}(e_1+e_2)=\inf  P_{u}(e_1)+\inf  P_{u}(e_2).
$

   Case (3): Let   $\big({{\mathrm{coz}}}(e_1) \big)'\in P$  and let ${{\mathrm{coz}}}(e_2)\in P$. 
  Similar to the proof of  Case (2), we have  
   \[ \begin{aligned}  
\inf  P_{u}(e_1+e_2)= \sup P_{l} (e_1+e_2)=1\in   P_{u}(e_1+e_2)\cap   P_{l}(e_1+e_2) 
\end{aligned} \]
and 
$\inf  P_{u}(e_1+e_2)=\inf  P_{u}(e_1)+\inf  P_{u}(e_2).
$

    Case (4): Let   $\big({{\mathrm{coz}}}(e_1) \big)'\in P$   and let  $\big({{\mathrm{coz}}}(e_2) \big)'\in P$. We have 
\[ \begin{aligned}  
   \big({{\mathrm{coz}}}(e_1+e_2) \big)'=\big({{\mathrm{coz}}}(e_1) \big)'\wedge  \big({{\mathrm{coz}}}(e_2) \big)'\in P 
&\Rightarrow {{\mathrm{coz}}}(e_1+e_2)\not\in P \\
&\Rightarrow P_{u}(e_1+e_2)\cap [0, 1)\cap \mathbb{Q}=\emptyset. 
\end{aligned} \] 
  If   
${{\mathrm{coz}}}(e_1)\wedge{{\mathrm{coz}}}(e_2)\in P$, 
then ${{\mathrm{coz}}}(e_1)\in P$ or  
${{\mathrm{coz}}}(e_2)\in P$, which implies that \linebreak
$\top={{\mathrm{coz}}}(e_1)\vee \big({{\mathrm{coz}}}(e_1) \big)'\in P$ 
or 
$\top={{\mathrm{coz}}}(e_2)\vee \big({{\mathrm{coz}}}(e_2) \big)'\in P$,  a contradiction. Hence, 
$P_{u}(e_1+e_2)\cap [1, 2)\cap \mathbb{Q}=\emptyset  $, 
which follows that 
$P_{u}(e_1+e_2)= [2, +\infty)\cap \mathbb{Q}.$ 
  A similar proof shows that  $ P_{l}(e_1+e_2)=(-\infty, 2]\cap \mathbb{Q}$.   
  Therefore,    
\[ \begin{aligned}  
\inf  P_{u}(e_1+e_2)= \sup P_{l} (e_1+e_2)=2\in   P_{u}(e_1+e_2)\cap   P_{l}(e_1+e_2) 
\end{aligned} \]
and 
$\inf  P_{u}(e_1+e_2)=\inf  P_{u}(e_1)+\inf  P_{u}(e_2).
$
\end{proof}
\begin{corollary} \label{prime90}
Let 
$P $   
be a prime ideal  of 
$L$.   
Suppose that $e_1, e_2\in B(\mathcal{R}(L))$.
If  \linebreak 
$x=\inf P_{u}(e_1\diamond e_2)$, 
then 
$x\in R_{e_1\diamond e_2}$
for every $\diamond\in\{\cdot, \vee, \wedge\}.$
\end{corollary}
\begin{proof}
By Proposition \ref{prime20},  it is obvious.
\end{proof}
We recall that for every $p\in \mathbb{R},$
\[ \begin{aligned}
  \mathbf{p}(-,r)=  
  \begin{cases}
\bot, & r\leq p, \\
\top,   & r > p,
\end{cases} \,\,\, \text{and} \,\,\, \,\,\,
  \mathbf{p}(r,-)=  
  \begin{cases}
\top ,& r< p ,\\
\bot,   & r  \geq p.
\end{cases}
 \end{aligned} \]
A function $f(x)$ is said to be bounded on a given interval  if there exist two numbers $M$ and $N$
such that: $M\leq f(x)\leq N.$  Also,    $\mathcal{R}^*(L)$ is the set of bounded elements of $\mathcal{R}(L).$ 
\begin{lemma}   \label{prime25}
Let $P $   be a prime ideal 
 of $L$.  If $\alpha\in  \mathcal{R}^*(L)$, then 
 $  
\inf  P_{u}(\alpha)= \sup P_{l} (\alpha).  
$ 
\end{lemma}
\begin{proof}
There exists $n\in \mathbb{N}$ such that  $  \mathbf{-n}\leq  \alpha\leq \mathbf{n}$, which implies that 
\[ \begin{aligned}    
&\top= \mathbf{n}(-,x)\leq \alpha(-,x) \,\, \text{ for every } \,\, x>n
\,\, \text{ and } \,\,  \\ & 
\top= \mathbf{-n}(y,-)\leq \alpha(y,-) \,\, \text{ for every } \,\, y<-n.
\end{aligned} \]
Hence, $(n,+\infty)\cap \mathbb{Q}\subseteq  P_{u}(\alpha)$ and 
 $(-\infty, -n)\cap \mathbb{Q}\subseteq  P_{l}(\alpha)$.  
It follows from Corollary \ref{prime16} that 
 $\inf  P_{u}(\alpha)= \sup P_{l} (\alpha).$  
\end{proof}
\begin{proposition} \label{prime30}
Let $P $   be a prime ideal 
 of $L$ and let $n\in\mathbb{N}$.   If $e_1, \ldots , e_n\in B(\mathcal{R}(L))$, then 
$$  
\inf  P_{u}(\mathbf{r_1}e_1+ \cdots + \mathbf{r_n}e_n)= \sup P_{l} (\mathbf{r_1}e_1+ \cdots +\mathbf{ r_n}e_n)  
$$ 
belongs to 
$$
  P_{u}(\mathbf{r_1}e_1+ \cdots + \mathbf{r_n}e_n)\cap  P_{l}(\mathbf{r_1}e_1+ \cdots + \mathbf{r_n}e_n)
$$ 
and 
\[ \begin{aligned}
\inf  P_{u}(\mathbf{r_1}e_1+ \cdots + \mathbf{r_n}e_n)=r_1\inf  P_{u}(e_1)+ \cdots + r_n\inf  P_{u}(e_n)
\end{aligned} \]
 for every $r_1, \ldots , r_n\in \mathbb{Q}$. 
\end{proposition}
\begin{proof} 
We prove it by induction on  $n\in\mathbb{N}$. 
For $n = 1$, the result is trivial.
Let $r_1, \ldots , r_n\in \mathbb{Q}$ be given and let $n\geq 2$. 
We put $\alpha=\mathbf{r_1}e_1+ \cdots + \mathbf{r_{n-1}}e_{n-1}$.  
Let $x\in P_{u}(\alpha+\mathbf{r_n}e_n )$ be given. Then 
\[ \begin{aligned} 
\bigvee_{a \in \mathbb{Q}}\alpha(a,-)\wedge \mathbf{r_n}e_n (x-a,-)=(\alpha+\mathbf{r_n}e_n )(x,-)\in P , 
\end{aligned} \]
which implies that $\alpha(a,-)\wedge \mathbf{r_n}e_n (x-a,-)\in P$  for every  $a \in \mathbb{Q}$, 
and so  
$\alpha(a,-) \in P$ or  $ \mathbf{r_n}e_n (x-a,-)\in P$ for every  $a \in \mathbb{Q}$. 
If $a<\inf  P_{u}(\alpha)$, then $\alpha(a,-) \not\in P$ , which implies from Remark \ref{remark3} that 
\begin{center}
$\bot= \mathbf{r_n}e_n (x-a,-)\in P$ or 
${\mathrm{coz}}(e_n)'  = \mathbf{r_n}e_n (x-a,-)\in P$. 
 \end{center} 
  From 
 \[ \begin{aligned}
 v:= \mathbf{r_n}e_n (x-\inf  P_{u}(\alpha),-)
 = \bigvee_{\substack{a \in \mathbb{Q}\\ a<\inf  P_{u}(\alpha) }}
  r_ne_n (x-a,-),
\end{aligned} \]
we conclude that 
$v=\bot \in P$
or
$v=\mathrm{coz}(e_n)'  \in P$, 
which implies that  \linebreak 
$\inf  P_{u}(\mathbf{r_n}e_n )\leq x-\inf  P_{u}(\alpha)$
and hence,  
\[ \begin{aligned}
\inf  P_{u}(\alpha)+ \inf  P_{u}(\mathbf{r_n}e_n )\leq 
 \inf  P_{u}(\alpha+\mathbf{r_n}e_n ).
\end{aligned} \]

Let  $x\in P_{l}(\alpha+\mathbf{r_n}e_n )$ be given. 
We set 
 \[ \begin{aligned}
 w:= \mathbf{r_n}e_n (-,x-\sup  P_{l}(\alpha))
 = \bigvee_{\substack{s \in \mathbb{Q}\\ s>\sup  P_{l}(\alpha) }}
  r_ne_n (-, x-s). 
\end{aligned} \]
Therefore, similar to the above, it conclude that 
 for every $s>\sup   P_{l}(\alpha)$, which we conclude from  Remark \ref{remark3} that  
$w=\bot \in P$ or $w= {\mathrm{coz}}(e_n)  \in P$.
Thus we have $x-\sup  P_{l}(\alpha)\leq  \sup  P_{l}( \mathbf{r_n}e_n)$ and hence,  
 \[ \begin{aligned}
\sup  P_{l}(\alpha+\mathbf{r_n}e_n )
\leq 
\sup  P_{l}(\alpha)+ \sup  P_{l}( \mathbf{r_n}e_n), 
\end{aligned} \]
 and finally, by Lemma  \ref{prime25},  
\[ \begin{aligned}
\inf  P_{u}(\alpha)+ \inf  P_{u}(\mathbf{r_n}e_n )=
 \inf  P_{u}(\alpha+\mathbf{r_n}e_n ).
\end{aligned} \]
The rest follows with Lemma  \ref{prime21}.  
\end{proof}
\section{${{\mathrm{Coz}}}\big(B( \mathcal{C}_c(L) )\big)$ as a base for ${{\mathrm{Coz}}}_c[L]$
} 
In \cite{EstajiTaha1},
for every 
$r\in \mathbb{R}$, it was defined that 
$u_r:=\bigvee_{n\in  \mathbb{N}}(-, a_n)$ 
and 
$l_r:=\bigvee_{n\in  \mathbb{N}}(b_n, -),$
where 
${a_n}, {b_n}\subseteq \mathbb{Q}$
are sequences such that 
$a_n, b_n\rightarrow r$
and  for every 
$n\in \mathbb{N} ,$
$$a_n < a_{n+1}< r < b_{n+1} <  b_n.$$
It was also shown that 
$\alpha(u_r)$
and
$\alpha(l_r)$
are complemented in 
$L$
and that the complement of
$\alpha(l_r)$
is
$\alpha(u_r)$. Also, see Lemma 4.1 in \cite{estaj2}.

 
The following results are counterpart of \cite[Proposition 5.1]{estaj2}.
\begin{proposition}\label{counterpart of proposition 5.1}
If 
$\alpha \in \mathcal {C}_c( L),$
then 
$\mathrm{coz}(\alpha)$
is a countable join of complemented elements of
$L$.
\end{proposition}
\begin{proof}
For any
$n\in,\mathbb{N}$,
by  \cite[Remark 7.2 ]{EstajiTaha1},
there exist $r_n,s_n\in \mathbb{R}\smallsetminus R_\alpha$
satisfying all of the following conditions:
\begin{enumerate}
\item[{\em (1)}] 
$\frac{-1}{n} < r_n < 0 < s_n <  \frac{1}{n},$
\item[{\em (2)}] 
$u_{r_n}, u_{s_n}, l_{r_n}$ 
and
$l_{s_n}$
are complemented,
\item[{\em (3)}] 
$\alpha(\frac{1}{n} , -) \leqslant  \alpha( l_{s_n}) \leqslant \mathrm{coz}(\alpha)$
and
$\alpha(- ,\frac{-1}{n}) \leqslant  \alpha(u_{r_n}) \leqslant \mathrm{coz}(\alpha),$ 
\item[{\em (4)}] 
$\alpha(u_{r_n}\vee l_{s_n})$
is complemented.
\end{enumerate}
Therefore,
$$\mathrm{coz}(\alpha)=\bigvee _{n\in \mathbb{N}} \alpha(- ,\frac{-1}{n})\vee \alpha(\frac{1}{n} , -)\leqslant  \alpha(u_{r_n})\vee \alpha( l_{s_n}) \leqslant \mathrm{coz}(\alpha),$$
which shows that 
$\mathrm{coz}(\alpha)=\bigvee _{n\in \mathbb{N}}\alpha(u_{r_n} \vee  l_{s_n}) .$
\end{proof}
\begin{remark}\label{111111}
It is easy to see that if $a$ is complemented in $L$ and $\alpha\in \mathcal R (L)$, then \linebreak$\beta:  \mathcal L(\mathbb R) \longrightarrow  L$ given by
\[
\beta(p,q)=
\left \{
\begin{array}{lll}
\alpha(p,q)\vee a' & \hspace{7mm}\mbox{if $\,  \,\, 0\in \tau(p, q)$},\\[2mm]
\alpha(p,q)\wedge a &\hspace{7mm}\mbox{if $\,  \,\, 0\not\in \tau(p, q)$},\\[2mm]
\end{array}
\right.
\]
is a continuous real-valued function.
Since
$\beta$
is a continuous real-valued function,
and by
\cite[Proposition 4.1]{EstajiAbedi},
$e^2_a=e_a$ and
$\mathrm{coz}(e_\alpha)=a$.
It is easily seen that $e$
is an idempotent element of
$\mathcal{R} (L)$
is and only if 
$e=e_{coz(\alpha)}$.

\end{remark}
The following results are counterpart of   \cite[Theorem 5.2]{estaj2}.

\begin{proposition}\label{counterpart of theorem 5.2 }
For every
$\alpha\in \mathcal{C}_c(L) ,$
there exists a sequence
$\{\alpha_n\}_{n\in \mathbb{N}}\subseteq B( \mathcal{C}_c(L) )$ with
${\mathrm{coz}}(\alpha)=\bigvee_{n\in\mathbb{N}} {\mathrm{coz}}(\alpha_n)$.
\end{proposition}
\begin{proof}
By Proposition \ref{counterpart of proposition 5.1} and  Remark \ref{111111}, it is obvious.
\end{proof}


\begin{definition}\label{c-completely}
A frame $L$
is a \textit{ $c$-completely regular frame} if and only if 
$\mathrm{Coz}_c[L]$
is a base for
$L$.
\end{definition}
The following results are counterpart of    \cite[Proposition 4.4]{Ghader(2013)}.
\begin{proposition}\label{counterpart of proposition 4.4}
Let  
$L$
be a frame. Then  
$L$ 
is a zero-dimensional frame if and only if
$L$ 
is a $c$-completely regular frame. 
\end{proposition}
\begin{proof}
\textit{Necessity.}
Let 
$a\in L$.
Then there exists a sequence 
$\{a_\lambda\}_{\lambda\in\Lambda}$
of complemented elements of
$L$
such that 
$a=\bigvee_{\lambda\in\Lambda}a_\lambda.$
Also, we know
${\mathrm{coz}}(e_{a_\lambda})=a_\lambda$,
which implies that
 $a=\bigvee_{\lambda\in\Lambda} {\mathrm{coz}}(e_{a_\lambda}).$
Since for every
$\lambda\in\Lambda, e_{a_\lambda}$
  is an idempotent element  and all the idempotent elements belong to
$\mathcal{C}_c(L)$,
the proof is complete.

 \medskip

\textit{Sufficiency.}
By Definition \ref{c-completely},
${\mathrm{Coz}}_c[L]$
is a base for
$L$,
and by
Proposition \ref{counterpart of proposition 5.1},
for every
$\alpha \in \mathcal {C}_c(L),$
${\mathrm{coz}}(\alpha)$
is a countable join of complemented elements of
$L$.
Then
$L$ 
is a zero-dimensional frame.
\end{proof}

\section{${\mathrm{Coz}}_c[L]$ as a $\sigma$-frame 
} 
In this section, the main purpose of writing this article  is going to be proved, that is, ${{\mathrm{Coz}}}_c[L]$ is a $\sigma$-frame for every completely regular frame $L$. To achieve this purpose, we need some lemmas and propositions, which we will express and prove them at the beginning.
\begin{lemma}  \label{prime35}
Let $\alpha$ be an element of $ \mathcal{R}(L)$ and let $x\in \mathbb{R}$. Then the following statements are  true:
\begin{enumerate} 
\item[{\rm (1)}] 
 If $x\in R_{\alpha}$, then  there exists a  prime ideal $P$ 
 in $L$ such that $x=\inf P_{u}(\alpha)$. 
 \item[{\rm (2)}]   	If there exists a countably $\vee$-complete   prime ideal $P$ 
 in $L$ such that \linebreak 
 $x=\inf P_{u}(\alpha)$, then 
 $x\in R_{\alpha}$. 
 \item[{\rm (3)}]   
Let $P $   be a prime ideal  of $L$ and let $n\in\mathbb{N}$ be given.   
 Suppose that $e_1, \ldots , e_n\in B(\mathcal{R}(L))$. 
 If $r_1, \ldots , r_n\in \mathbb{Q}$ and  $x=\inf P_{u}(\mathbf{r_1}e_1+ \cdots + \mathbf{r_n}e_n)$, then   \linebreak 
 $x\in R_{\mathbf{r_1}e_1+ \cdots + \mathbf{r_n}e_n}$. 
\end{enumerate}
\end{lemma}
\begin{proof} 
 (1).   
By the hypothesis we have  ${{\mathrm{coz}}}(\alpha-\mathbf{x})\not =\top$. Then  there exists  a  prime ideal $P$ 
 in $L$  such that 
 $ {{\mathrm{coz}}}(\alpha-\mathbf{x})\in P$, which implies that  
 $$\bigvee_{\substack{r\in\mathbb{Q}\\r<x }} \alpha(-,r) \vee 
  \bigvee_{\substack{r\in\mathbb{Q}\\r>x }} \alpha( r,-) \in P. 
  $$
  Therefore, 
 \[ \begin{aligned} 
\forall   r\in\mathbb{Q}\,\,\,\big( r<x \Rightarrow  \alpha(-,r)\in P\Rightarrow r\in P_{l}(\alpha) \big)  \Rightarrow x\leq \sup P_{l}(\alpha),
  \end{aligned} \]
 and 
  \[ \begin{aligned} 
\forall  r\in\mathbb{Q}\,\,\,\big( r>x 
  \Rightarrow  \alpha(-,r)\in P
  \Rightarrow r\in P_{u}(\alpha) \big)   \Rightarrow x\geq \inf P_{u}(\alpha).
 \end{aligned} \]
 Since, by Corollary \ref{prime16}, $\sup P_{l}(\alpha)=\inf P_{u}(\alpha)$,  
  we conclude that $x= \inf P_{u}(\alpha).$ 
\medskip  

(2). 
 Let $P$ be a  prime ideal of $L$  such that $x= \inf P_{u}(\alpha).$ 
 Then, by Lemma \ref{prime15},  
 \[ \begin{aligned} 
\forall  r\in\mathbb{Q}\,\,\,\big( r<x 
 \Rightarrow  r\in  P_{l}(\alpha)
 \Rightarrow  \alpha(-,r)\in P \big)  
  \end{aligned} \]
 and 
  \[ \begin{aligned} 
 \forall  r\in\mathbb{Q}\,\,\,\big( r>x 
 \Rightarrow  r\in  P_{u}(\alpha)
  \Rightarrow  \alpha(r,-)\in P \big)  .
 \end{aligned} \]
By the hypothesis we have  
  \[ \begin{aligned} 
 {{\mathrm{coz}}}(\alpha-\mathbf{x})=
 \bigvee_{\substack{r\in\mathbb{Q}\\r<x }} \alpha(-,r) \vee 
  \bigvee_{\substack{r\in\mathbb{Q}\\r>x }} \alpha( r,-) \in P, 
 \end{aligned} \] 
 and so  $x\in R_{\alpha}$.
 
 (3). By Lemma \ref{prime35}, 
 $$
 x\in P_{u}(\mathbf{r_1}e_1+ \cdots + \mathbf{r_n}e_n)\cap  P_{l}(\mathbf{r_1}e_1+ \cdots + \mathbf{r_n}e_n).
 $$
Then 
 $$
 {{\mathrm{coz}}}(\mathbf{r_1}e_1+ \cdots + \mathbf{r_n}e_n-x)=(\mathbf{r_1}e_1+ \cdots + \mathbf{r_n}e_n)(-,x)\vee (\mathbf{r_1}e_1+ \cdots + \mathbf{r_n}e_n)(x,-)\in P,
 $$ 
 which implies that 
 ${{\mathrm{coz}}}(\mathbf{r_1}e_1+ \cdots + \mathbf{r_n}e_n-\mathbf{x})\not =\top. $    
  Therefore,  $x\in R_{\mathbf{r_1}e_1+ \cdots + \mathbf{r_n}e_n}$. 
\end{proof}
For every $A,B\subseteq \mathbb R$, we put
$A\diamond B:=\big\{a\diamond b \mid    a\in A \,\, and\,\, b\in B\big\}$ 
for every $\diamond\in \{+,\cdot, \wedge, \vee\}$. 
\begin{proposition} \label{prime40}
Let $n$ be an element of   $ \mathbb{N}$.   
 Suppose that $e_1, \ldots , e_n\in B(\mathcal{R}(L))$. 
 If
 \linebreak $r_1, \ldots , r_n\in \mathbb{Q}$, then   
 $$  
 R_{\mathbf{r_1}e_1+ \cdots + \mathbf{r_n}e_n}\subseteq 
  R_{\mathbf{r_1}e_1}+ \cdots + R_{\mathbf{ r_n}e_n}.
 $$   
\end{proposition}
\begin{proof}
Suppose that $x\in R_{\mathbf{r_1}e_1+ \cdots + \mathbf{r_n}e_n}$. 
Thus, by Lemma \ref{prime35}(1), 
  there exists a  prime ideal $P$ 
 in $L$ such that $x=\inf P_{u}(\mathbf{r_1}e_1+ \cdots + \mathbf{r_n}e_n)$. 
 Hence,  by Proposition \ref{prime30}, 
$x=x_1+\cdots +x_n$, where $x_i=\inf P_{u}(\mathbf{r_i}e_i)$ 
for every $1\leq i\leq n$.  
Then, by Lemma \ref{prime35}(3), $x_i\in R_{\mathbf{r_i}e_i}$. 
Therefore, 
$$  
 R_{\mathbf{r_1}e_1+ \cdots +\mathbf{ r_n}e_n}\subseteq 
  R_{\mathbf{r_1}e_1}+ \cdots + R_{ \mathbf{r_n}e_n}.
 $$   
\end{proof}
\begin{proposition} \label{prime95}
Suppose that 
$e_1, e_2 \in B(\mathcal{R}(L)).$ 
Then, 
$R_{e_1\diamond e_2}\subseteq  R_{e_1}\diamond R_{e_2}$
for every \linebreak$\diamond\in\{\cdot, \vee, \wedge\}.$
\end{proposition}
\begin{proof}
By Proposition
\ref{prime20},
Corollary
\ref{prime90},
and Lemma
\ref{prime35},
this proposition will be easily proved.
\end{proof}
\begin{remark} \label{prime45} 
If $\alpha$  and $\beta$ are  elements of $ \mathcal{R}(L)$ with 
 $\alpha\leq\beta$, then 
 $P_{u}(\beta)\subseteq P_{u}(\alpha)$ and\linebreak 
 $P_{l}(\alpha)\subseteq P_{l}(\beta)$.
\end{remark}

Let $(c_n)_{n=0}^{\infty}\in \mathrm{Coz}[L]$ and 
$(\alpha_n)_{n=0}^{\infty}\in \mathrm{Coz}[L]$ with 
$ \mathbf{0}\leq \alpha_n \leq \mathbf{1}$ and $\mathrm{coz}(\alpha_n)=~c_n$ for every $n$, be given. 
Banaschewski \cite{Banaschewski2009}
showed that $\alpha:=\sum_{i=0}^{\infty} \frac{\alpha_i}{\mathbf{2}^i}\in\mathcal{R}(L)$, \linebreak 
 $\mathrm{coz}(\alpha)=\bigvee_{i=0}^{\infty} c_n$ 
  and   $\alpha\leq \sum_{i=0}^{n} \frac{\alpha_i}{\mathbf{2}^i} + 
\frac{\mathbf{1}}{ \mathbf{2}^n}$  for every $n$.

Given any countable family $(e_n)_{n=0}^{\infty}$ in 
$B\big(\mathcal{R}(L)\big)$, consider 
$\alpha=\sum_{n=0}^{\infty} \frac{e_n}{\mathbf{2}^n}$.
 Then for every $n\in \mathbb{N}$,  we have
\[ \begin{aligned}   
\alpha\leq \sum_{i=0}^{n} \frac{e_i}{\mathbf{2}^i} + 
\frac{\mathbf{1}}{ \mathbf{2}^n} 
\Rightarrow 
 P_{u}(\sum_{i=0}^{n} \frac{e_i}{\mathbf{2}^i} + \frac{\mathbf{1}}{ \mathbf{2}^n}  )\subseteq P_{u}(\alpha)  
\Rightarrow 
\inf  P_{u}(\alpha)\leq \sum_{i=0}^{n} \inf P_{u}(\frac{e_i}{\mathbf{2}^i})+ \frac{\mathbf{1}}{ \mathbf{2}^{n}} ,
\end{aligned} \]
which implies that 
$ \inf  P_{u}(\alpha)\leq \sum_{i=0}^{\infty} \inf P_{u}(\frac{e_i}{\mathbf{2}^i})$. 
On the other hand,  for every $n\in \mathbb{N}$,  it follows that
\[ \begin{aligned}   
\sum_{i=0}^{n} \frac{e_i}{\mathbf{2}^i}  \leq \alpha 
\Rightarrow
  P_{u}(\alpha)  \subseteq  
 P_{u}(\sum_{i=0}^{n} \frac{e_i}{\mathbf{2}^i})
\Rightarrow 
\sum_{i=0}^{n} \inf P_{u}(\frac{e_i}{\mathbf{2}^i}) 
 \leq \inf  P_{u}(\alpha) ,
\end{aligned} \]
which implies that 
$ \sum_{i=0}^{\infty} \inf P_{u}(\frac{e_i}{\mathbf{2}^i})  \leq \inf  P_{u}(\alpha) $. 
 Hence,  $ \inf  P_{u}(\alpha)=\sum_{i=0}^{\infty} \inf P_{u}(\frac{e_i}{\mathbf{2}^i})$.
\begin{proposition} \label{prime50}  
Let $(e_n)_{n=0}^{\infty}$  be a   countable family  of    $B\big(\mathcal{R}(L)\big)$. 
 If $\alpha=\sum_{n=0}^{\infty}  \frac{e_n} {\mathbf{2}^n}$, 
 then   
$  
 R_{\alpha}\subseteq 
 \sum_{n=0}^{\infty} R_{ \frac{e_n}{\mathbf{2}^n}}.
$
\end{proposition}
\begin{proof}
Suppose that $x\in R_{\alpha}$. 
Thus, by Lemma \ref{prime35}(1), 
  there exists a  prime ideal $P$ 
 in $L$ such that $x=\inf P_{u}(\alpha)$. 
 Hence,  it follows from  Remark \ref{prime45} that 
$x= \sum_{n=0}^{\infty} x_i$, where\linebreak $x_i=\inf P_{u}(\frac{e_i}{\mathbf{2}^i})$ 
for every $ i\in \mathbb{N}\cup\{0\}$.  
Then, by Lemma \ref{prime35}(3), we have $x_i\in R_{\frac{e_i}{\mathbf{2}^i}}$. 
Therefore, 
$ 
 R_{\alpha}\subseteq 
 \sum_{n=0}^{\infty} R_{ \frac{e_n}{\mathbf{2}^n}}.
$  
\end{proof}
\begin{theorem} \label{prime55}  
 For every completely regular frame $L$, 
   ${{\mathrm{Coz}}}_c[L]$ is a $\sigma$-frame. 
\end{theorem}
\begin{proof}
By Proposition \ref{counterpart of theorem 5.2 }, it suffices to show that 
if $(e_n)_{n=0}^{\infty}$  is a   countable family  of    $B\big(\mathcal{R}(L)\big)$, then 
$\bigvee_{n=0}^{\infty} {{\mathrm{coz}}}(e_n)\in {{\mathrm{Coz}}}_c[L]$. 
We put $\alpha=\sum_{n=0}^{\infty} \frac{e_n}{\mathbb{{\mathbf{2}}}^n}$. 
It is well known that  
$ {{\mathrm{coz}}}(\alpha)=\bigvee_{n=0}^{\infty} {{\mathrm{coz}}}(e_n)$. 
Since 
$ \sum_{n=0}^{\infty} R_{ \frac{e_n}{\mathbf{2}^n}}$ is a  countable set, we conclude from Proposition \ref{prime50}  
that $ R_{\alpha}$ is too, which implies that 
 $\alpha\in C_c(L)$. 
 Therefore, $\bigvee_{n=0}^{\infty} {{\mathrm{coz}}}(e_n)\in {{\mathrm{Coz}}}_c[L]$. 
\end{proof}
An element $a$ of frame $L$ is said to \textit{ rather below} an  element $b$ or $a$  is {\it well inside}  $b$, written as $a\prec b$, if there is an element $x$ such that $a\wedge x=\bot$ and $x\vee b=\top.$
In \cite{Charalambous}, a $\sigma$-frame $L$ is \textit{regular} if for every $a\in L$, there exist $(a_n)_{n=1}^{\infty}$ such that $a_n\prec a$ and $a=\bigvee_{n=1}^{\infty} a_n$ for every $n$. 
\begin{corollary}\label{regular} 
For every completely regular frame $L$, 
${{\mathrm{Coz}}}_c[L]$ is regular.
\end{corollary}  
A bounded lattice $L$ is \textit{normal} if given 
$a \vee b = \top$ that  $a,b\in L,$  there exist $u, v\in L$ such that
$a\vee u = \top = v\vee b$ and $u \wedge v =\bot.$
Also, a normal $\sigma$-frame $L$ is  \textit{perfectly normal} if for each $a \in L$ there exist $(a_n)_{n=1}^{\infty}$ such that   $a\wedge b=a\wedge c$ if and only $a_n\vee b=a_n\vee c$ for every $a,b\in L$ (see \cite{Charalambous}).

By \cite[Proposition 1-1]{Christopher},  it is obvious that  
${{\mathrm{Coz}}}_c[L]$ is normal and perfectly normal for every completely regular frame $L$.

An \textit{Alexandroff algebra} is a normal $\sigma$-lattice $L$ that
for every $a\in L$, there exist $(a_n, b_n)_{n=1}^{\infty}$
such that  
$b_n\vee a =\top$, 
$ b_n\wedge a_n =\bot $ and 
$\bigvee_{n=1}^{\infty} a_n=a.$
\begin{corollary} 
For every completely regular frame $L$, 
${{\mathrm{Coz}}}_c[L]$ is an Alexandroff algebra frame.
\end{corollary}  
\begin{proof} 
Let $a\in L$ be given. 
By Proposition \ref{counterpart of theorem 5.2 }, 
 there exists $(e_n)_{n=1}^{\infty}\subseteq B\big(\mathcal{R}(L)\big)$ such that 
 $a=\bigvee_{\lambda\in\Lambda} {\mathrm{coz}}(e_{a_\lambda}).$ We 
put $(a_n, b_n)_{n=1}^{\infty}=\big({\mathrm{coz}}(e_n),{\mathrm{coz}}(e_n)'\big)_{n=1}^{\infty}$, then  we have
\begin{enumerate} 
\item[{\rm (1)}] 
$\bigvee_{n=1}^{\infty} a_n=a,$
\item[{\rm (2)}]  
$ b_n\wedge a_n =\bot $  for every $n\in \mathbb N$,  and 
\item[{\rm (3)}] 
$b_m\vee a={\mathrm{coz}}(e_m)'\vee\bigvee_{n=1}^{\infty}{\mathrm{coz}}(e_n)=\top$   for every $m\in \mathbb N$.
\end{enumerate} 
Therefore, the proof is complete.
\end{proof}
A \textit{cover} of   $\sigma$-frame  $L$ is a 
family $ (a_n)_{n=1}^{\infty} \subseteq  L $ such that $\bigvee_{n=1}^{\infty} a_n=\top.$
A cover $(b_n)_{n=1}^{\infty}$ of   $\sigma$-frame  $L$ is \textit{locally finite} if there
is a cover $(c_n)_{n=1}^{\infty}$ of $L$ and, for each $m$, a finite subset $p(m)$ of $\mathbb{N}$ such that $c_m \wedge b_n = \bot$
if $n\notin p(m)$. 
A $\sigma$-frame is \textit{paracompact} if for each countable cover $\mathcal{A}=(a_n)_{n=1}^{\infty}$ of $L$ there
is a locally finite cover $\mathcal{B}= (b_m)_{m=1}^{\infty}$ of $L$ such that for each $m$ there is some $n$ with
$b_m \leq a_n$. 
A cover $(a_n)_{n=1}^{\infty}$ of   $\sigma$-frame  $L$ is \textit{shrinkable} if there is a
cover $(d_n)_{n=1}^{\infty}$ of $L$ with $d_n \leq a_n$ for every $n$.

\begin{proposition}
For every completely regular frame $L$, the following statements are  true: 

\begin{enumerate} 
\item[{\rm (1)}] 
${{\mathrm{Coz}}}_c[L]$ is paracompact.
\item[{\rm (2)}] 
Each cover of ${{\mathrm{Coz}}}_c[L]$ is shrinkable.
\end{enumerate} 
\end{proposition}
\begin{proof}
By \cite[ Proposition 4]{Banaschewski(1989)} and Corollary \ref{regular},  it is obvious.
\end{proof}

\end{document}